\documentclass[reqno,11pt,centertags]{article}
\usepackage{amsmath,amsthm,amscd,amssymb,latexsym,upref}
\date{\today}
\usepackage[T2A,OT1]{fontenc}
\usepackage[russian,english]{babel}
\usepackage[margin=3.5 cm]{geometry}

\usepackage{hyperref}
\hypersetup{
    colorlinks=true, %set true if you want colored links
    linktoc=all,     %set to all if you want both sections and subsections linked
    linkcolor=blue,  %choose some color if you want links to stand out
}

\newcommand{\bbD}{{\mathbb{D}}}
\newcommand{\bbE}{{\mathbb{E}}}

\newcommand{\bbC}{{\mathbb{C}}}

\newcommand{\bbT}{{\mathbb{T}}}

\newcommand{\z}{\zeta}

\newcommand{\Res}{\text{\rm Res}\,}

\allowdisplaybreaks \numberwithin{equation}{section}
\newtheorem{theorem}{Theorem}[section]

\newtheorem{lemma}[theorem]{Lemma}

\theoremstyle{definition}
\newtheorem{definition}[theorem]{Definition}

\newtheorem{assumption}[theorem]{Assumption}

\def\be{\begin{equation}}
\def\ee{\end{equation}}
\def\bea{\begin{eqnarray}}
\def\eea{\end{eqnarray}}
\def\bean{\begin{eqnarray*}}
\def\eean{\end{eqnarray*}}

\def\restr#1{\,\vrule\,\lower1ex\hbox{$#1$}}

\def\a{\alpha}

\def\d{\delta}
\def\D{\Delta}

\def\g{\gamma}
\def\G{\Gamma}

\def\l{\lambda}
\def\L{\Lambda}

\def\z{\zeta}

\title
{Cauchy Integral Formula for Fuchsian Groups. II}
\author{Alexander Kheifets\thanks{The author is thankful to S. Denisov for useful comments
and to the reviewers for the suggestions that helped to improve the manuscript.}}

\begin{document}

\maketitle

%%%%\begin{center}
%%%%\textit{Dedicated to the memory of Professor Harry Dym}
%%%%\end{center}
%%%%\medskip

%%% \tableofcontents

\begin{abstract}
We prove Conjecture 4.1 from \cite{K} (Theorem \ref{15} below):
a generalization of the Hasumi's Direct Cauchy Theorem property
for the derivatives.
This proof substitutes all proofs in \cite{K}.
\end{abstract}

\section{Introduction}
For the reader's convenience we repeat here some notations and definitions from \cite{K}.
Further references can be found in\cite{K}.
Let $D$ be a bounded domain in $\bbC$. Let $\bbD$ be the unit disk and let
$\bbT$ be the unit circle. Let $\L:\bbD\to D$ be the uniformization map and let $\G$ be the
corresponding Fuchsian group on $\bbD$.
\begin{definition}
Let $\a$ be a character of a Fuchsian group $\G$ on $\bbD$. We say that a function $u$
defined on $\bbD$ or/and on $\bbT$ is $\a$-automorphic if
$$
u\circ\g=\a(\g) u
$$
for every $\g\in\G$.
\end{definition}
\begin{definition}
We say that function $u(\z)$ analytic on $\bbD$ is of bounded characteristic if it is a
ratio of two bounded analytic functions $u(\z)=\dfrac{u_1(\z)}{u_2(\z)}$. We say that $u(\z)$
is of Smirnov class if the denominator is an outer function. We say that $u(\z)$ is an outer
Smirnov class function in both $u_1(\z)$ and $u_2(\z)$ are outer functions.
\end{definition}
Let $g_\z$ be the (complex) Green function of $\G$ with respect to point $\z\in\bbD$.
That is, $g_\z$ is the Blaschke product with zeros at the orbit of $\z$ under $\G$.
Assume that $\G$ is of Widom type, that is, that
$g'_\z$ is of bounded characteristic.
Using a Frostman theorem,
%%%(see, e.g. Appendix in \cite{KhJC})
one can write 
\footnote{
In this formula in \cite{K} (the first formula on Page 2) the factor of $1-|\z|^2$ was missed. Consequently, 
there should be $\dfrac{1-|\z|^2}{|t-\z|^2}$ instead of $\dfrac{1}{|t-\z|^2}$
in many places throughout \cite{K}. Namely, in formulas (1.2), (1.3), (2.2), (2.3), (2.6), (3.2), (3.3),
in some unnumbered formulas in the proofs of Theorems 1.6, 2.1, 3.1,
in Definition 3.6 and Conjecture 4.1.
}
for almost every $t\in\bbT$
\be\label{12}
\dfrac{g'_\z(t)}{g_\z(t)}=\sum\limits_{\g\in\G}\dfrac{1-|\z|^2}{|\g(t)-\z|^2}\dfrac{\g'(t)}{\g(t)}.
\ee
Also, by a Pommerenke theorem
%(see Theorem 4 in \cite{Pom}),
\be\label{11}
g'_\z=\dfrac{\D_\z}{\psi_\z},
\ee
where $\D_\z$ is an inner function and $\psi_\z$ is a bounded outer function.
Moreover, $g_\z$ and $\D_\z$ are character-automorphic functions.
We denote their characters as $\mu_\z$ and $\d_\z$, respectively.
%%%%,
%%%%that is, there exist characters $\mu_\z$ and $\d_\z$ of the group $\G$ such that
%%%%$$
%%%%g_\z\circ\g=\mu_\z(\g)g_\z\quad\text{and}\quad
%%%%\D_\z\circ\g=\d_\z(\g)\D_\z, \quad \g\in\G.
%%%%$$
\begin{definition}
We say that analytic on $D$ function $h$ belongs to $H^1(D)$ if
$h\circ\Lambda$ is a Smirnov class function of $\bbD$ and
$$
\int\limits_{\partial D}|h(s)||ds|<\infty
$$
\end{definition}
\begin{definition}
We say that the Cauchy Integral Formula holds for the domain $D$ if
for every function $h\in H^1(D)$
\be\label{C241229-01}
\dfrac{1}{2\pi i}\oint\limits_{\partial D}\dfrac{h(s)}{s-\l}ds=h(\l),\quad \l\in D.
\ee
\end{definition}
Throughout this paper we will make the following
\begin{assumption}\label{C250322-01}
$\L'$ and
$
\dfrac{\L(t)-\L(\z)}{g_\z(t)}
$
are outer Smirnov class functions.
\end{assumption}

\section{Cauchy Integral Formula}
\begin{lemma}\label{03}
Let
\be\label{01}
\widetilde h(t)
=
h(\L(t))
\dfrac{g_\z(t)^{k+1}}{(\L(t)-\L(\z))^{k+1}}
\dfrac{\L'(t)}{g'_\z(t)}\ ,
\ee
where 
$\L$, $g_\z$ are analytic on a neighborhood of $\z$
$$
g_\z(\z)=0,\quad g'_\z(\z)\ne 0,\quad \L'(\z)\ne 0,
$$
and $h$ is analytic on a neighborhood of $\L(\z)$. Then
$$
%%\dfrac{1}{2}
h^{(k)}(\L(\z))
=
%%%\dfrac{1}{2}
\left(
\left(
\dfrac{1}{g'_\z(t)}
\dfrac{d}{dt}
\right)^k
\widetilde h(t)
\right)_{| t=\z}.
$$
\end{lemma}
\begin{proof}
Since $\L'(\z)\ne 0$, there exist a neighborhood $V_1$ of $\z$ and a neighborhood $U_1$ of $\L(\z)$
such that $\L$ maps conformally (one-to-one) $V_1$ onto $U_1$. Let $C_1$ be a circle centered at $\z$
that lies in $V_1$. Then
$$
\dfrac{h^{(k)}(\L(\z))}{k!}
=
\dfrac{1}{2\pi i}
\oint\limits_{\L(C_1)}
\dfrac{h(s)}{(s-\L(\z))^{k+1}}
ds
$$
\be\label{07}
=
\dfrac{1}{2\pi i}
\oint\limits_{C_1}
h(\L(t))
\dfrac{\L'(t)}{(\L(t)-\L(\z))^{k+1}}
dt
=
\Res_\z
\left(
h(\L(t))
\dfrac{\L'(t)}{(\L(t)-\L(\z))^{k+1}}
\right).
\ee
We rewrite relation \eqref{01} as
\be\label{04}
h(\L(t))
\dfrac{\L'(t)}{(\L(t)-\L(\z))^{k+1}}
=
\widetilde h(t)
\dfrac{g'_\z(t)}{g_\z(t)^{k+1}}.
\ee
It follows from \eqref{04} that the residue in \eqref{07} is equal to
\be\label{08}
\Res_\z
\left(
h(\L(t))
\dfrac{\L'(t)}{(\L(t)-\L(\z))^{k+1}}
\right)
=
\Res_\z
\left(
\widetilde h(t)
\dfrac{g'_\z(t)}{g_\z(t)^{k+1}}
\right).
\ee
Since $g'_\z(\z)\ne 0$, there exist a neighborhood $V_2$ of $\z$ and a neighborhood $U_2$ of $0$ such that $g_\z$ maps conformally (one-to-one) $V_2$ onto $U_2$. Then we can view $\widetilde h(t)$ as
$$
\widetilde h(t)=\widetilde{\widetilde h}(g_\z(t)),\quad t\in V_2.
$$
Let $C_2$ be a circle centered at $\z$
that lies in $V_2$. Then the residue in the right-hand side of \eqref{08} equals
$$
\Res_\z
\left(
\widetilde h(t)
\dfrac{g'_\z(t)}{g_\z(t)^{k+1}}
\right)
=
\dfrac{1}{2\pi i}
\oint\limits_{C_2}
\widetilde{\widetilde h}(g(t))
\dfrac{g'_\z(t)}{g_\z(t)^{k+1}}
dt
$$
$$
=
\dfrac{1}{2\pi i}
\oint\limits_{g_\z(C_2)}
\dfrac{\widetilde{\widetilde h}(s)}{s^{k+1}}
ds
=
\dfrac{\widetilde{\widetilde h}^{(k)}(0)}{k!}
=
\dfrac{1}{k!}
\left(
\left(
\dfrac{1}{g'_\z(t)}
\dfrac{d}{dt}
\right)^k
\widetilde{\widetilde h}(g_\z(t))
\right)_{|t=\z}
$$
\be\label{10}
=
\dfrac{1}{k!}
\left(
\left(
\dfrac{1}{g'_\z(t)}
\dfrac{d}{dt}
\right)^k
\widetilde h(t)
\right)_{|t=\z}.
\ee
Comparing \eqref{07}, \eqref{08}, and \eqref{10} we get the assertion of the lemma.
\end{proof}
\begin{theorem}\label{15}
Let $D$ be a bounded domain in $\bbC$ for which the Cauchy Integral Formula \eqref{C241229-01} holds.
Assume that the uniformization map $\L(\z)$ meets Assumption \ref{C250322-01}. Let $\G$ be
the corresponding Fuchsian group on $\bbD$. Then for every
$\mu^k_\z\d_\z$ automorphic $H^1(\bbD)$ function $f$ we have
$$
\int\limits_{\bbT}
\dfrac{f(t)}{\D_\z(t)g_\z(t)^k}
\dfrac{1-|\z|^2}{|t-\z|^2}L(dt)
=
\dfrac{1}{k!}
\left(
\left(
\dfrac{1}{g'_\z(t)}
\dfrac{d}{dt}
\right)^k
\left(\dfrac{f(t)}{\D_\z(t)}\right)
\right)_{| t=\z}
,\quad \z\in\bbD,
$$
where $g_\z$ is the Green function, $\D_\z$ is the inner part of $g'_\z$ $($see formula \eqref{11}$)$,
$\mu_\z$ is the character of $g_\z$, $\d_\z$ is the character of $\D_\z$,
and $L(dt)$ is the normalized Lebesgue measure on $\bbT$.
\end{theorem}
\begin{proof}
%%%Similar to what was done in \cite{K}, 
We start with
the Cauchy Integral Formula for $D$ differentiated $k$ times
%% see Definition 1.3 of \cite{K})
$$
\dfrac{1}{2\pi i}\oint\limits_{\partial D}\dfrac{h(s)}{(s-\l)^{k+1}}ds=
\dfrac{1}{k!}h^{(k)}(\l), \quad h\in H^1(D), \quad \l\in D.
$$
We do
the uniformization substitution $s=\L(t)$
$$
\dfrac{1}{k!}h^{(k)}(\L(\z))=
\dfrac{1}{2\pi i}\int\limits_{\bbE}\dfrac{h(\L(t))\L'(t)}{(\L(t)-\L(\z))^{k+1}}dt,
$$
where $\bbE$ is the fundamental set of $\G$ on $\bbT$, $\z\in \bbD$.
We can rewrite the latter as
$$
\dfrac{1}{k!}h^{(k)}(\L(\z))
=
\dfrac{1}{2\pi i}\int\limits_{\bbE}\dfrac{h(\L(t))\L'(t)}{(\L(t)-\L(\z))^{k+1}}
\dfrac{g_\z(t)}{g'_\z(t)}
\dfrac{g'_\z(t)}{g_\z(t)}
dt
$$
by formula \eqref{12}
\be\label{13}
=
\dfrac{1}{2\pi i}
\sum\limits_{\g\in\G}
\int\limits_{\bbE}
\dfrac{h(\L(t))\L'(t)}{(\L(t)-\L(\z))^{k+1}}
\dfrac{g_\z(t)}{g'_\z(t)}
\dfrac{1-|\z|^2}{|\g(t)-\z|^2}\dfrac{\g'(t)}{\g(t)}
dt.
\ee
Observe that
\be\label{14}
\dfrac{\L' g_\z}{g'_\z}\circ\g
=
\dfrac{(\L\circ\g)'}{\g'}
\dfrac{\mu_\z(\g) g_\z}{(g_\z\circ\g)'}\g'
=
\L'
\dfrac{\mu_\z(\g) g_\z}{\mu_\z(\g) g'_\z}
=
\L'
\dfrac{g_\z}{g'_\z}.
\ee
Therefore, we may continue from \eqref{13} as follows
$$
=
\dfrac{1}{2\pi i}
\sum\limits_{\g\in\G}
\int\limits_{\bbE}
\dfrac{h(\L(\g(t)))\L'(\g(t))}{(\L(\g(t))-\L(\z))^{k+1}}
\dfrac{g_\z(\g(t))}{g'_\z(\g(t))}
\dfrac{1-|\z|^2}{|\g(t)-\z|^2}\dfrac{\g'(t)}{\g(t)}
dt
$$
$$
=
\dfrac{1}{2\pi i}
\sum\limits_{\g\in\G}
\int\limits_{\g(\bbE)}
\dfrac{h(\L(t))\L'(t)}{(\L(t)-\L(\z))^{k+1}}
\dfrac{g_\z(t)}{g'_\z(t)}
\dfrac{1-|\z|^2}{|t-\z|^2}\dfrac{dt}{t}
$$
$$
=
\dfrac{1}{2\pi i}
\int\limits_{\bbT}
\dfrac{h(\L(t))\L'(t)}{(\L(t)-\L(\z))^{k+1}}
\dfrac{g_\z(t)}{g'_\z(t)}
\dfrac{1-|\z|^2}{|t-\z|^2}\dfrac{dt}{t}
%%%%=
%%%%\int\limits_{\bbT}
%%%%\dfrac{h(\L(t))\L'(t)}{(\L(t)-\L(\z))^{k+1}}
%%%%\dfrac{g_\z(t)}{g'_\z(t)}
%%%%\dfrac{L(dt)}{|t-\z|^2}
%%%%,
$$
\be\label{05}
=
\int\limits_{\bbT}
\dfrac{h(\L(t))\L'(t)}{(\L(t)-\L(\z))^{k+1}}
\dfrac{g_\z(t)}{g'_\z(t)}
\dfrac{1-|\z|^2}{|t-\z|^2}L(dt)
%%%=
%%%\dfrac{1}{k!}h^{(k)}(\L(\z))
.
\ee
In view of \eqref{14},
$$
\dfrac{h(\L(t))\L'(t) }{(\L(t)-\L(\z))^{k+1}}
\dfrac{g_\z(t)}{g'_\z(t)}
$$
is an automorphic (with the trivial character) function.
To have a Smirnov class function we define
\be\label{06}
f(t)=\dfrac{h(\L(t))\L'(t) }{(\L(t)-\L(\z))^{k+1}}
\dfrac{g_\z(t)}{g'_\z(t)}
g_\z(t)^k \D_\z(t).
\ee
$f$ is indeed of Smirnov class if $\L$
meets Assumption \ref{C250322-01} and $h\in H^1(D)$.
Moreover, in this case $f\in L^1(\bbT)$, since
\be\label{C250323-03}
\int\limits_{\bbT}
|f(t)|
\dfrac{1-|\z|^2}{|t-\z|^2}L(dt)
=
\int\limits_{\bbT}
\left|
\dfrac{f(t)}{\D_\z(t)}
\right|
\dfrac{1-|\z|^2}{|t-\z|^2}L(dt)
=
\dfrac{1}{2\pi}\oint\limits_{\partial D}\dfrac{|h(s)|}{|s-\l|}|ds|
.
\ee
Therefore, $f\in H^1(\bbD)$.
$f$ is also $\mu^k_\z\d_\z$ automorphic. 
Conversely, for an arbitrary function $f\in H^1(\bbD)$ that is automorphic with the character
$\mu^k_\z\d_\z$ one can recover $H^1(D)$ function $h$ via formula \eqref{06}.

In terms of $f$ formula \eqref{05} reads as
$$
\int\limits_{\bbT}
\dfrac{f(t)}{\D_\z(t)g^k_\z(t)}
\dfrac{1-|\z|^2}{|t-\z|^2}L(dt)
=
\dfrac{1}{k!}
h^{(k)}(\L(\z))
=
\dfrac{1}{k!}
\left(
\left(
\dfrac{1}{g'_\z(t)}
\dfrac{d}{dt}
\right)^k
\left(\dfrac{f(t)}{\D_\z(t)}\right)
\right)_{| t=\z}.
$$
The latter equality is due to Lemma \ref{03} with
$
\widetilde h=\dfrac{f}{\D_\z}.
$
\end{proof}

\noindent
{\bf Funding:\ }
This work was not supported by any funding.

\bigskip

\bigskip

A. Kheifets, Department of Mathematics and Statistics, University of Massachusetts Lowell, One University Ave.,
Lowell, MA 01854,USA

\emph{E-mail address:} {Alexander\underline{ }Kheifets@uml.edu}

\end{document}